\documentclass[a4paper,11pt]{amsart}
\usepackage{hyperref,latexsym}
\usepackage{enumerate}
\usepackage{graphicx}

\usepackage{xcolor}

\theoremstyle{plain}
\newtheorem{theorem}{Theorem}[section]
\newtheorem{lemma}[theorem]{Lemma}
\newtheorem{corollary}[theorem]{Corollary}
\newtheorem{proposition}[theorem]{Proposition}
\theoremstyle{definition}
\newtheorem{definition}[theorem]{Definition}
\newtheorem{example}{Example}
\theoremstyle{remark}
\newtheorem{remark}{Remark}

\newcommand{\M}{\mathcal{M}}

\begin{document}

\title[ Magnetic Billiards]
      {Magnetic billiards: Non-integrability for strong magnetic field; Gutkin type examples}

\date{November 2019}
\author{Misha Bialy, Andrey E. Mironov, Lior Shalom}
\address{M. Bialy and L. Shalom, School of Mathematical Sciences, Tel Aviv
University, Israel} \email{bialy@tauex.tau.ac.il, shalom@mail.tau.ac.il}
\address{A.E. Mironov, Sobolev Institute of Mathematics, Novosibirsk, Russia } \email{mironov@math.nsc.ru}
\thanks{M.B. and L.S. were supported in part by the Israel Science Foundation grant}
\thanks{ A.M. was supported by Mathematical Center in Akademgorodok
	Novosibirsk State University, Pirogova street, 2, 630090, Novosibirsk, Russia
	and
	Sobolev Institute  of Mathematics,  4 Acad. Koptyug avenue, 630090, Novosibirsk, Russia.}


\begin{abstract}  We consider 
	magnetic billiards under a strong constant magnetic field. The purpose of this paper is two-folded. We examine the question of existence of polynomial integral of billiard magnetic flow. As in our previous paper \cite{BM0} we succeed to reduce this question to algebraic geometry test on existence of polynomial integral, which shows polynomial non-integrability for all but finitely many values of the magnitude.  
	In the second part of the paper we construct examples of magnetic billiards which have the so called $\delta$-Gutkin property, meaning that any Larmor circle entering the domain with angle $\delta$ exits the domain with the same angle $\delta$. For ordinary Birkhoff billiard in the plane such examples were introduced by E. Gutkin and are very explicit. Our construction of Gutkin magnetic billiards relies on beautiful examples by F.Wegner of the so called Zindler curves, which are related to the problem of floating bodies in equilibrium, which goes back to S.Ulam (Problem 19 in Scottish book \cite{scottish}). We prove that Gutkin magnetic billiard can be obtained as a parallel curve to a Wegner curve. Wegner curves can be written by elliptic functions in polar coordinates so the construction of magnetic Gutkin billiard is rather explicit but much more complicated. 
	
\end{abstract}

\maketitle

%%%%%%%%%%%%%%%%%%%%%%%%%%%%%%%%%%%%%%%%%%%%%%%%%%%%%%%%%%%%%%%%%%%%%%%%%%

%%%%%%%%%%%%%%%%%%%%%%%%%%%%%%%%%%%%%%%%%%%%%%%%%%%%%%%%%%%%%%%%%%%%%%%%%%

%%%%%%%%%%%%%%%%
\section{\bf Introduction and main results}

In this paper we consider
a magnetic billiard inside a domain
$\Omega\subset\mathbb{R}^2$ bounded by a simple smooth
closed curve $\gamma$. The  magnetic field is assumed to be of
constant magnitude $\beta>0$,
so the
particle moves inside $\Omega$ with unit speed along an arc of a Larmor circle (always oriented counterclockwise)
of constant radius $$r=\frac{1}{\beta}.$$ Upon hitting the boundary $\gamma=\partial \Omega$, the particle is reflected
according to the law of geometric optics. This model is called
magnetic Birkhoff billiard. We denote by $g^t$ the  magnetic billiard flow, i.e. the flow of unit tangent vectors to billiard trajectories.

Magnetic
Birkhoff billiards were introduced by Robnik and Berry \cite{BR} and  studied in many papers
\cite{Berg}, \cite{B}, \cite{GB}, \cite{k}, 
\cite{T2}.

In this paper we shall assume that the magnitude of the magnetic field is relatively large with respect to the curvature $k$ of the boundary curve $\gamma$. More precisely, let $r_0$ be the maximal possible radius for a tubular neighborhood of $\gamma$  (which is diffeomorphic to the normal bundle of $\gamma$). In particular, one has
$$r_0\leq\frac{1}{\max |k|}.$$ Our assumption on the magnetic field is such that the radius of Larmor circles $r$ satisfy:
\begin{equation}\label{assumption}
r<\frac{r_0}{2}\leq \frac{1}{2 \max |k|}.
\end{equation}
In particular
$$
 {\max}|k|<\beta/2.
$$
Under this assumption every Larmor arc entering the domain transversally to the boundary cannot be tangent to $\gamma$ at the exit point, and thus gets reflected according to the billiard rule. Moreover this assumption assures that the parallel curves to 
$\gamma$ at the distance up to $2r$ are all smooth (see Section 2).
Notice, that the assumption does not require the domain to be convex.
We refer to \cite{BR} for the discussion of the dynamical behavior of three possible regimes-
weak, intermediate and strong magnetic field.
\begin{figure}[h]
	\centering
	\includegraphics[width=0.5\linewidth]{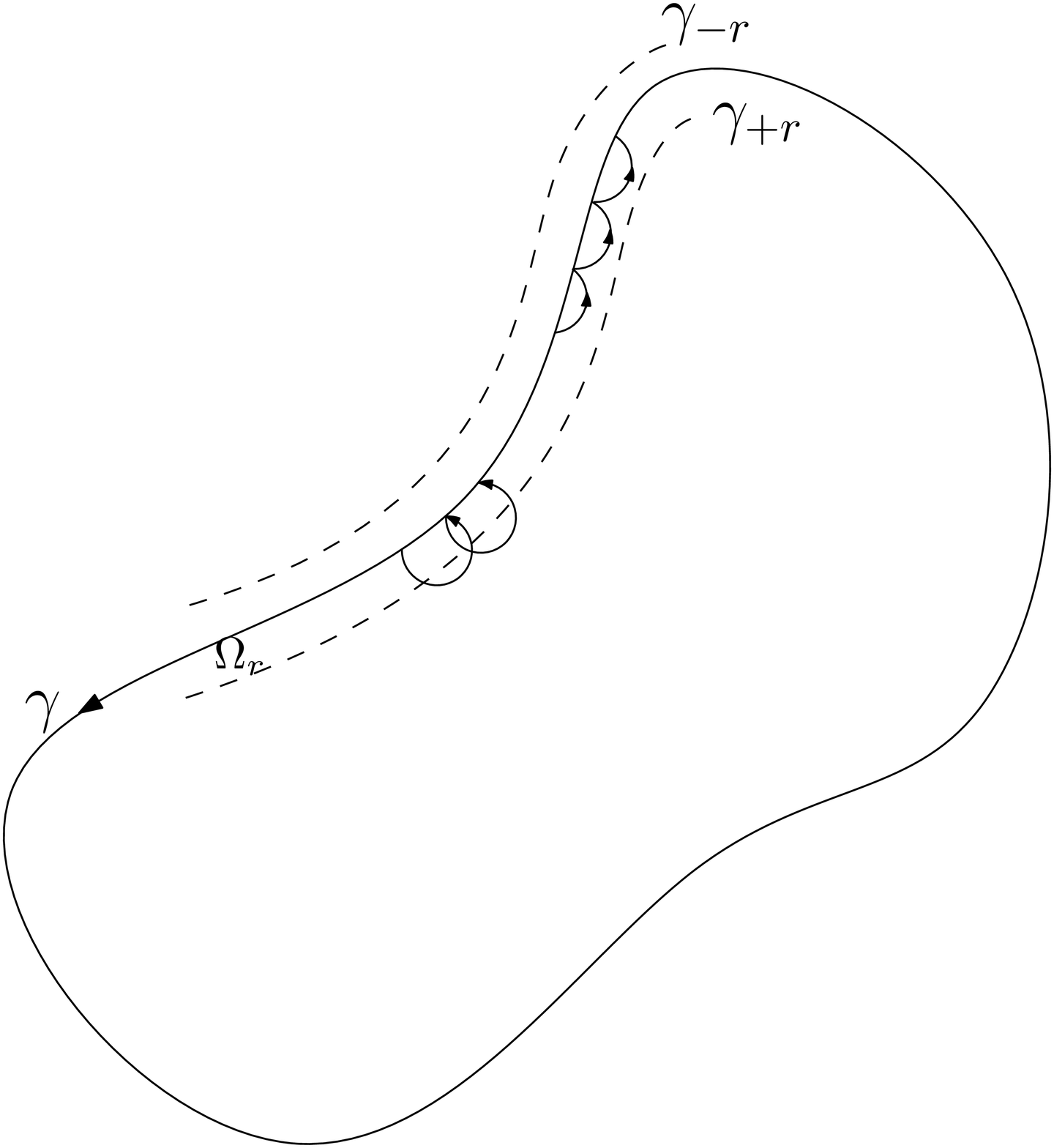}
	\caption{Magnetic billiard.}
	\label{1}
\end{figure}
Recently, in \cite{BM0,royal} we examined this problem in an
algebraic setting for small magnitudes of the magnetic field $\beta<\min k(x)$, using ideas from our recent papers on ordinary
Birkhoff billiards \cite{BM1}, \cite{BM2} (which in turn were influenced by previous
results of \cite{Bolotin} and \cite{Tab}).

 In \cite{BR}, the computer evidence
of chaotic regions of a magnetic billiard inside an ellipse is
demonstrated for all magnitudes $\beta$ of the magnetic field. In particular for $\beta >0$, unlike the case $\beta=0$, the pictures show that the magnetic billiard is not integrable. Recently even better pictures appeared
in \cite{albers}.

We present two new results. First result states the non-existence of polynomial integrals for magnetic billiard flow in a strong field. 
This is the extension of the results and techniques obtained in \cite{BM0} for the case of weak magnetic fields.

The second result proves the existence of Gutkin type magnetic billiards.  
The key ingredient of our approach to both results is to consider the "dual" object to magnetic billiard, namely the magnetic billiard map $\mathcal M$. Map $\mathcal M$ acts on the domain $\Omega_r$ which consists of all centers of Larmor circles intersecting the boundary curve $\gamma$. This domain is the annulus bounded by two smooth parallel curves $\gamma_{\pm r}$ and is a natural phase space of the magnetic billiard. Moreover, $\mathcal M$ is a symplectic diffeomorphism of $\Omega_r$, transforming Larmor center to the next one (see Section \ref{S2} below).

\begin{example}\label{example}Let $\gamma$ be a circle  centered at the
	origin. Then the function which measures the distance of the center
	of Larmor circle to the origin is invariant under reflections and
	hence is an integral $h$ of the billiard flow $g^t$. Specifically,
	$h$ has the form:
	$$h(x,v)=x_1^2+x_2^2+\frac{2}{\beta}(v_1x_2-v_2x_1).$$
	where $x\in\gamma$ and $v$ the unit inward vector at $x$.
\end{example}It is remarkable that there are no other examples known of
integrable magnetic billiards. Similarly to Birkhoff's conjecture for
ordinary billiards  (see \cite {royal} with references therein for recent progress in this conjecture), we suggest that the only integrable magnetic
billiard is the circular one. As usual, integrability can be
understood in various ways. In this paper we restrict to integrals which are
polynomial in the velocities of magnetic billiard flows in a strong magnetic field. Another
approach, that of the \emph{total} integrability, was
considered in \cite{B} for weak magnetic fields. It is an open question how to implement this approach for strong fields.

We turn now to the precise formulation of our results.
In the next theorem we assume that the  polynomial in momenta first integral is defined on the subset 
of the  phase space consisting of tangent vectors to all Larmor arcs lying inside $\Omega$ which intersect $\partial\Omega$ (see Definition 2.2 below). 
\begin{theorem}\label{main1}
	Let $\Omega$ be a  bounded domain with a smooth
	boundary $\gamma=\partial\Omega$ such that
	the curvature assumption (\ref{assumption}) is satisfied. Suppose that the
	magnetic billiard flow $g^t$ in $\Omega$ admits a non-constant polynomial in momenta integral $\Phi.$
	Then the curves $\gamma_{\pm r}$ are real ovals of affine algebraic curves which are non-singular $\mathbf{C}^2.$

\end{theorem}
\begin{corollary}\label{all1}
	For any non-circular domain $\Omega$ in the plane, the magnetic
	billiard inside $\Omega$ has no non-constant polynomial in momenta integral for all but
	finitely many values of $\beta$.
\end{corollary}
In many cases one can get non-integrability for all values of $\beta$, with no exception.
For instance, for ellipses the parallel curves appear to be singular curves of degree $8$ and therefore  we conclude:
\begin{corollary}\label{ellipse}Let $\Omega$ be the interior of the standard ellipse
	$$\partial\Omega=\left\{\frac{x^2}{a^2}+\frac{y^2}{b^2}=1\right\},\quad 0<b<a.$$
	Then for any magnitude of the magnetic field
	$\beta>2k_{\max}=\frac{2a}{b^2}$, the
	magnetic billiard in the ellipse does not admit a non-constant polynomial in momenta integral.
\end{corollary}

The method of the proof of Theorem \ref{main1}, is analogous to the one used in \cite{BM0}, but requires certain modifications.
The main difference is that for strong magnetic field arbitrary close to the boundary, there exist tangent vectors such that the motion of the particle does not touch the boundary but is a rotation along Larmor circle inside the domain. We shall give details in Section \ref{S2}.

Our second result deals with the Gutkin property. 
We say that a magnetic billiard satisfies $\delta$-Gutkin property, $0<\delta<\pi$, if  any Larmor arc $C$ entering  $\Omega$ with the angle $\delta$ with $\gamma$,  exits $\Omega$ with the same angle $\delta$ with $\gamma$ as well.  

Ordinary plane billiard tables with this property were characterized by E. Gutkin.
Gutkin found, that this property holds if and only if the following equation is satisfied:
%Explicit examples of billiard tables with this property, which can be obtained by  deformation of the circle were constructed for ordinary planar billiards by E. Gutkin.
%Gutkin found that for planar Gutkin billiards the number $\delta$ cannot be arbitrary but needs to satisfy the equation 
\begin{equation}\label{Gu}
\tan n\delta=n\tan\delta,
\end{equation}
	where %for some natural number% 
	$n$ denotes the index of a non-vanishing Fourier coefficient of the radius of curvature of $\gamma$. Also, it was proved by Van Cyr in \cite{VC} that $\delta$ is incommensurate with $\pi$, so in particular, one cannot choose any arbitrary $\delta$ for a Gutkin type billiard. In contrast with the planar case, it was proved in \cite {B2, B3} that the only Gutkin billiards in higher dimensions are round spheres, independently of the number $\delta\in(0,\pi/2)$. Tabachnikov et.al. \cite{????} studied Gutkin billiards  on the 2-sphere, where this problem was addressed infinitesimally. It is shown, that for an infinitesimal deformation of the circle the relation described by equation (\ref{Gu}) also appears. However, as far as we know, the existence of such a deformation on the sphere was not achieved in \cite {????}. 

Here we prove the following 
\begin{theorem}\label{G}
	For every $\delta\in(0,\pi)$ there exists a non-circular magnetic billiard in the plane with Gutkin property.
\end{theorem} 
 For the proof we use properties of the magnetic billiard map $\mathcal M$ and reduce the question to a very beautiful Wegner examples which provide solutions to the floating problem (S.Ulam problem number 19 of the Scottish book \cite{scottish}).
 On the other hand we tried the perturbation approach to the problem and prove that infinitesimally equation (\ref{Gu}) appears again. It is not clear to us if a genuine solution can be obtained this way.
 
 Proof and further details on Theorem \ref{G} are given in Section \ref{Gutkin}.

\section*{Acknowledgements}
It is a pleasure to thank Sergei Tabachnikov for references on Zindler and Wegner curves, and Lev Buhovsky for very fruitful discussions.

%%%%%%%%%%%%%%%%
 %%%%%%%%%%%%%%%%%%%%%%%%%%%%%%%%%%%
 \section {Preliminary notions and results}\label{S2}
 \subsection {Parallel curves}Let $s$ be an arc length parameter on $\gamma$, $J$ the complex structure on $\mathbb R^2$.
 Parallel curves are given by the formula:
 $$\gamma_t(s)=\gamma(s)+tJ\dot\gamma(s)=\gamma(s)+tn(s),$$
 where $n$ is the positive unit normal of $\gamma$. These curves are smooth embedded curves for $|t|<r_0$.
  \begin{remark}
 	The curves $\gamma_{\pm t}$ are also called equidistant curves, or
 	fronts, in singularity theory, or offset curves in computer aided
 	geometric design, CAGD, see \cite{SS}, \cite{SeS}.
 \end{remark}

 Denote by $k_{+r}$ and $k_{-r}$  the curvatures of
 the parallel curves $\gamma_{+r}$ and $\gamma_{-r}$ respectively.

	 \begin{lemma}\label{bounds}
	
	 \begin{enumerate}
	 	\item The curvature of the parallel curve $\gamma_t$ satisfies
	 	$$k({\gamma_t})=\frac{k(\gamma)}{1-t\cdot k(\gamma)}, $$ for all $t,  \text{ s.t } |t|<r_0$.
	 	\item 	For $t=\pm r$ we have:$$ |k_{\pm r}|<\beta.$$
	 \end{enumerate}
\end{lemma}
\begin{proof}
	The first item is an immediate computation.
	For the second item, notice that for $t=\pm r$ we have for the denominator from (\ref{assumption})
	$$
	1\pm rk ({\gamma})>1-r\max |k(\gamma)|>\frac{1}{2}
	$$
	Therefore, 
	$$
	|k_{\pm r}|<2 \max k(\gamma)<\frac{1}{r}=\beta,
	$$
	again by assumption (\ref{assumption}).
	\end{proof}
 This shows, in particular, that any circle of radius $r$ with the
 center at $Q=\gamma(s)$ is tangent  to the inner boundary $\gamma_{+r}$  from
 outside at the point $\gamma(s)+rJ\dot{\gamma}(s)$ and 
 to the outer boundary $\gamma_{- r}$ from inside
 at $\gamma(s)-rJ\dot{\gamma}(s)$.
 Apart
 from these tangencies, this circle remains entirely inside the annulus
 $\Omega_r$ (see Fig. \ref{3}).

\subsection{Phase space of magnetic billiard: cylinder vs. annulus}
 
 The phase space of a magnetic billiard can be understood in two ways. 
 
 The first way is the usual one for Birkhoff billiards. Namely, denote by $\mathcal C$ the collection of all Larmor arcs  intersecting the boundary $\gamma=\partial \Omega$. Magnetic billiard map $\mathcal B$ transformes a Larmor arc $C_-$ of the collection $\mathcal C$  to the next one $C_+$ obtained by the billiard reflection at the exit point of $C_-$ (see Fig. \ref{2}). 
 \begin{figure}[h]
 	\centering
 	\includegraphics[width=0.5\linewidth]{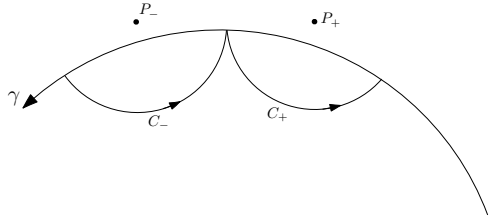}
 	\caption{Billiard reflection of $C_-$ to $C_+$.}
 	\label{2}
 \end{figure}
 The phase space can be identified with the space of inward unit vectors with foot point on the boundary $\gamma$, which is a cylinder with standard symplectic structure invariant under $\mathcal B$.
 
 The second appearance of the phace space is somewhat dual. 
Every Larmor circle intersecting $\gamma$ corresponds to a unique Larmor center. The set of all Larmor centers 
 fill the annulus which we denote by $\Omega_r$. The boundary of the annulus $\Omega_r$ consists of two parallel curves $\gamma_{+r}$,
 $\gamma_{-r}$ of $\gamma$, as in the previous subsection.

Moreover, the magnetic billiard map $\M:\Omega_r\rightarrow\Omega_r$ acts
by the following rule: Let $C_-(\epsilon)$  be an arc of Larmor circle exiting $\Omega$ at the point $Q$ with the angle $\epsilon$,  $\epsilon\in(0,\pi)$ (everything is oriented counterclockwise). Let $C_+(\epsilon)$ be the reflected arc. Then the Larmor
centers of $C_-(\epsilon)$, $C_+(\epsilon)$ are given by
\begin{equation}\label{a}
P_-(\epsilon)=Q+rJR_{-\epsilon}\dot\gamma(Q)=Q+rR_{\pi/2-\epsilon}\dot\gamma(Q),\quad
\end{equation}
$$P_+(\epsilon)=Q+rJR_{\epsilon}\dot\gamma(Q)=Q+rR_{\pi/2+\epsilon}\dot\gamma(Q),$$
where $R_\epsilon$ is a counterclockwise rotation by angle $\epsilon$.

Now it is clear how mapping $\mathcal M$
acts. 
Let  $C$ be the circle of radius $r$ centered at $Q$. It is tangent to the curves $\gamma_{+ r}$ and  $\gamma_{- r}$ at the points $P'$, $P''$ respectively. Then the points $P_-(\epsilon)$ and $P_+(\epsilon)=:\mathcal M(P_-)$ 
lie on $C$ and in a symmetric way with respect to $P'P''$ (see Fig.\ref{3}).

\begin{figure}[h]
	\centering
	\includegraphics[width=0.5\linewidth]{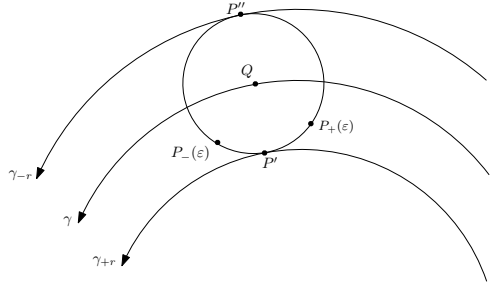}
	\caption{Larmor centers $P_{\pm}$ lie symmetrically on the circle centered at $Q\in\gamma$.}
	\label{3}
\end{figure}
The map $\M:\Omega_r\rightarrow\Omega_r$ preserves the
standard symplectic form in the plane (see \cite{BM0} for the proof), and thus $\Omega_r$ naturally
becomes the phase space of the magnetic Birkhoff billiard.  Notice that on the boundaries
$\gamma_{\pm r}$, the map $\M$ acts as the identity map.
\subsection{Polynomial integrals}
Let us remark that all possible arcs of the collection $\mathcal C$ (i.e. all Larmor arcs lying inside $\Omega$ which intersect $\partial\Omega$) sweep 
the annulus bounded by the curves $\gamma$ and $\gamma_{+2r}$ which we denote by $\Omega_{+}$.
We are concerned with the existence of first integrals polynomial in
the velocities for magnetic billiard flow.
\begin{definition}\label{def} Let $\Phi:T_{x}\Omega_{+}\rightarrow\mathbb{R}$
	be a function on the unit tangent bundle, $\Phi=\sum_{0\leq k+l\leq N}
	a_{kl}(x)v_1^kv_2^l$, which is a polynomial in the components $v_1,v_2$ of $v$
	with coefficients
	which are smooth up to the boundary, $a_{kl}\in C^\infty(\overline\Omega_{+}).$ We
	call $\Phi$ a polynomial integral of the magnetic billiard flow if the
	following conditions hold:
	
	1. $\Phi$ is an integral of the magnetic flow $g^t$ between the collisions, i.e. for every $x\in\gamma$ and any unit inward vector $v$ at $x$ we have
	$$\Phi(g^t(x,v))=\Phi(x,v) $$
	
	2. $\Phi$ is preserved under the reflections of the
	boundary $\partial\Omega$, i.e, for any point
	$x\in\partial\Omega,$
	$$\Phi(x,v)=\Phi(x,v-2\langle n,v\rangle n),$$ for any $v\in T_x\Omega \text{ s.t } |v|=1,$
	where $n$ is the unit normal to $\partial\Omega$ at $x$.
\end{definition}

Let us denote by $\mathcal L$ the mapping assigning the center of Larmor circle and any unit tangent vector to the circle:
\begin{equation}\label{L}
\mathcal L(x,v)=x+rJv.
\end{equation}
Given a polynomial integral $\Phi=\sum_{ 0\leq k+l\leq N}
a_{kl}(x)v_1^kv_2^l$ of the magnetic billiard, we define the
function $F:\Omega_r\rightarrow\Omega_r$ by the requirement
\begin{equation}\label{F}
F\circ{\mathcal L}=\Phi.
\end{equation}
 This is a well-defined construction, since $\Phi$ is an
integral of the magnetic flow, and therefore takes constant values
on any Larmor circle. Moreover, since $\Phi$ is invariant under the
billiard flow, $F$ is invariant under the billiard map $\M$:
$$
F\circ\M=F.
$$

% In
%coordinates, the definition (\ref{F}) reads
%\begin{equation}\label{v} F(x_1-rv_2,x_2+rv_1)=\Phi(x,v)=\sum_{0
%\leq k+l\leq N}a_{kl}(x)v_1^kv_2^l.
%\end{equation}
 Furthermore, we claim the following fact.
 \begin{theorem}\label{smooth}
 	Function $F$ is $C^{\infty}$-smooth on $\overline{\Omega_{r}}$.
 \end{theorem}
  We prove Theorem \ref{smooth} in Section \ref{proofs}.
 Notice that by the very definition of function $F$ it has a remarkable property:
 
 {\it Function F restricted to any circle $$C_s:=\{x: |x-\gamma(s)|=r\}$$ is a restriction to $C_s$ of a polynomial in two variables of degree at most $N$ (and hence is a trigonometric polynomial of degree at most $N$}).
 
 To see this, fix a point $z=\gamma(s)$ and take all possible unit tangent vectors $v$ at $z$. The corresponding $y=\mathcal L(z,v)$ form the circle $C_s$ defined above.
 By definition, $F$ restricted to $C_s$ has the form 
 
 $$
 F(y)= \Phi(z, \frac{1}{r}J(z-y)),
 $$
 so it is the restriction of a polynomial function in $y$ of degree at most $N$, hence a trigonometric polynomial on $C_s$.

 Next we state the following, the proof is given in Section \ref{proofs}:
 \begin{theorem}\label{harmonics}
 	Let $F\in C^\infty(\overline{\Omega_{r}})$ is such that its restriction to any circle $C_s$ is the restriction of a  polynomial of degree at most $N$. It then follows that $F$ is a polynomial function in the variables $(x_1,x_2)$ of degree at most $2N$. 
 \end{theorem}
 
 \begin{remark}
It is plausible that the statement of Theorem \ref{harmonics} remains valid for lower regularity, say for continuous functions. It was indeed the case for magnetic billiards in weak magnetic fields \cite{BM0}. Also, we conjecture that the complete circles  in formulation of Theorem \ref*{harmonics} can be replaced with arcs of radius $r$. If this conjecture holds, one can extend Theorem \ref*{harmonics} to functions $F\in C(\Omega_r)$, which means that the coefficients  of the first integral $\Phi$ in Definition \ref{def} can be assumed only continuous. 
 \end{remark}

Moreover, we have the following:
\begin{proposition}\label{prop}
Suppose that the magnetic billiard in $\Omega$ admits a polynomial
integral $\Phi$ and let $F$ be the corresponding polynomial on
$\overline{\Omega_r}$. Then 
$$F_|\gamma_{\pm r}={\rm const}.$$
\end{proposition}
We prove this proposition in section \ref{proofs}.
\begin{remark}\label{r}
We shall assume below that
the polynomial integral $F$ of $\M$ is such that the ${\rm
constants}$ in Proposition \ref{prop} is $0,$ for both parallel
curves $\gamma_{\pm r}$. Indeed, if $F_|\gamma_{- r}=c_1$ and
$F_|\gamma_{+ r}=c_2,$ one can replace $F$ by
$F^2-(c_1+c_2)F+c_1\cdot c_2$ to annihilate both constants
$c_1,c_2.$ 
\end{remark}

Proposition \ref{prop} and Remark \ref{r} imply that
$$\gamma_{\pm r}\subset \{F=0\},$$
and thus $\gamma_{\pm r}$ is contained in the algebraic curve
$\{F=0\}$. This fact implies then, that $\gamma$ itself is algebraic.
In the proof of our main theorem we denote by $f_{\pm r}$ the minimal defining
polynomials of the irreducible component in $\mathbb{C}^2$
containing $\gamma_{\pm r}$ respectively. Since the curves
$\gamma_{\pm r}$ are real, $f_{\pm r}$ have real coefficients.
Notice that it may happen that both $\gamma_{\pm r}$ belong to the
same component, so that $f_{+r}=f_{-r}.$ For instance, this is the
case for parallel curves to $\gamma$ when $\gamma$ is an ellipse
\cite {FN}, \cite{SS}, \cite{SeS}. In this case $f_{-r}=f_{+r}$ is
an irreducible polynomial of degree $8$.

%%%%%%%%%%%%%%%%%%%%%%%%%%%%%%%%%%%%%

%%%%%%%%%%%%%%%%%%%%%%%%%%%%%%%%%%%%%%%%%%%%%%%%%%%%%%%%%%%%%%%

%%%%%%%%%%%%%%%%%%%%%%%%%%%%%%%%%%%%

\section{Gutkin magnetic billiards}\label {Gutkin}
In this section we study magnetic Gutkin billiards.

\subsection{Geometric approach}
Our geometric construction uses the properties of the magnetic billiard map $\mathcal M$.
Namely, let $\Omega$ be a magnetic billiard domain with the Gutkin property corresponding to the angle $\delta$.
Denote by $\Gamma\subset\Omega_r$ the curve consisting of Larmor centers of all the arcs having angle $\delta$ with the boundary. Magnetic billiard has $\delta$-Gutkin property if and only if the curve $\Gamma$ is invariant under $\mathcal M$.
\begin{theorem}
	It then follows that $\Gamma$ is a Zindler curve (we use the terminology of \cite{LPT}). Namely, moving  the segment inscribed into $\Gamma$ of constant length 
	$$
	L=2r\sin\delta
	$$ 
	around $\Gamma$ is such that the velocity of the midpoint $M$ of the segment is necessarily parallel to the segment. The midpoints of these segments form the curve $\tilde\Gamma$ which is parallel to $\gamma$ at the distance $r\cos\delta$.
\end{theorem}
\begin{proof}
	Consider the triangle $\Delta P_-P'P_+$ (see Fig. \ref{4}). 
	It is isosceles with the angle $(\pi-\delta)$ at the vertex $P'$. This is because the arcs $P_-P'$ and $P'P_+$ of the circle centered at $Q$ both have central angle $\delta$, since the Larmor circle $C_-$ 
	is reflected to $C_+$ with the angle of incidence $\delta$. So $$|P_-P'|=|P'P_+|=2r\sin\delta /2.$$
	Notice that this triangle is rigid, hence moves isometrically when the segment moves around $\Gamma$.
	
	Next we need to consider the velocity vector of the midpoint $M$. 
	Notice that the segment $MP'$ has constant length 
	$$
	|MP'|= |P_-P'|\cos(\pi/2-\delta/2)=2r \sin^2(\delta/2),
	$$ 
	hence  the point $M$ moves over the parallel curve $\gamma_{t}$ where $$t=r-2r\sin^2\delta/2=r\cos\delta.$$
	
	Therefore, the velocity vector of $M$ is parallel to the tangent vector of $\gamma_{+ r}$ at $P'$ and thus  the midpoint $M$ moves orthogonally 
	to $P'M$, hence parallel to the segment. This is precisely Zindler property.
	
	In addition the distance to $\gamma$ from $\tilde\Gamma$ equals $r\cos\delta$.
\end{proof}
\begin{figure}[h]
	\centering
	\includegraphics[width=0.5\linewidth]{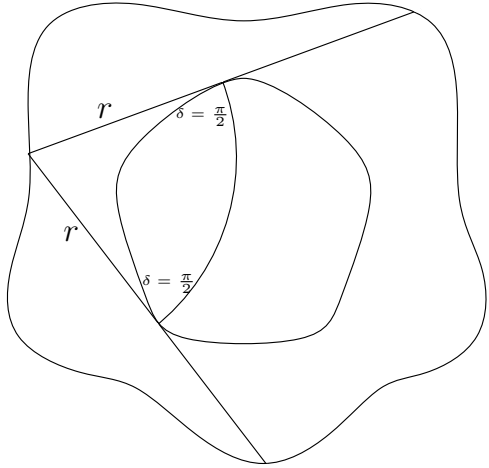}
	\caption{Magnetic billiard with Gutkin property for $\delta=\pi/2$.}
	\label{6}
\end{figure}
\begin{corollary}
	Suppose $\delta=\pi/2$. In this case $r\cos\delta=0$ and we have by the construction, that the curve $\tilde\Gamma$ coincides with $\gamma$. Thus for Wegner curve $\Gamma$ constructed for the length $2r$ of the inscribed segment, the curve of the midpoints of the segments is an example of Gutkin magnetic  billiard. This fact can be easily seen geometrically. (Fig.\ref{6})
\end{corollary}
This theorem gives the way to construct magnetic billiard domains with 
Gutkin property. Let $\Gamma$ be the Wegner curve constructed for the segment of length $
L=2r\sin\delta,
$
see \cite {Wegner}. $\Gamma$ determines the curve consisting of midpoints of the inscribed segments of length $
L=2r\sin\delta,
$ denote it by $\tilde\Gamma$. The last step is to reconstruct our magnetic billiard table  $\gamma$ as a curve parallel to $\tilde \Gamma$ at the distance $r\cos\delta$. 
\begin{figure}[h]
	\centering
	\includegraphics[width=0.9\linewidth]{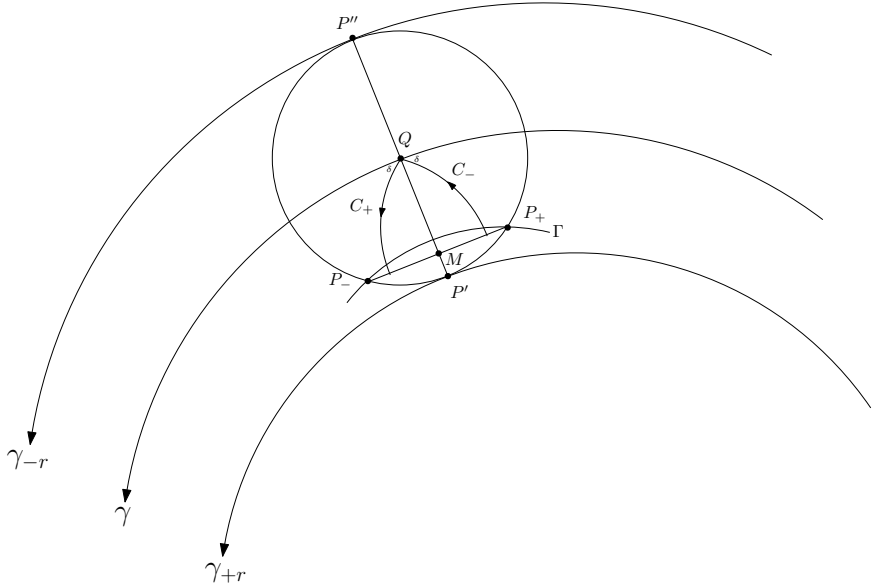}
	\caption{Curve $\Gamma$ of the Larmor centers.}
	\label{4}
\end{figure}

\subsection{Analytic approach, formal solution}
Choose parameterization on the curve $\gamma$ by the angle $\varphi$ of the tangent vector to $\gamma$ with a fixed direction. Let $\rho$ be the curvature radius of $\gamma$. Let $\gamma(\varphi_1)$, $\gamma(\varphi_2)$ be two points on $\gamma$, where $\varphi_1,\varphi_2 \in [0, 2\pi]$.

We assume that $\gamma$ has  $\delta$-Gutkin property.  Define the following change of variables: $$\bar{\varphi} = \frac{\varphi_2+\varphi_1}{2}, d = \frac{\varphi_2-\varphi_1}{2}.$$ Consider the invariant curve of the magnetic billiard consisting of the Larmor arcs starting with angle $\delta$ with $\gamma$ (invariance is equivalent to  $\delta$-Gutkin property). It is easy to see that along this invariant curve $d$ becomes a function of $\bar \varphi$.
Therefore we get that $\varphi_1 = \bar{\varphi} - d(\bar{\varphi})$ and $\varphi_2 = \bar{\varphi} + d(\bar{\varphi})$. So we  rewrite the $\delta$-Gutkin property of $\gamma$ as follows.
\begin{proposition}
	\begin{equation}\label{ro}
	\int_{\bar{\varphi} - d(\bar{\varphi})}^{\bar{\varphi} + d(\bar{\varphi})}\rho(\xi)e^{i\xi}d\xi = \frac{2}{\beta}\sin(\delta+d(\bar{\varphi}))e^{i\bar{\varphi}}
	\end{equation}
\end{proposition}

\begin{proof}We can get from $\gamma(\varphi_1)$ to $\gamma(\varphi_2)$ through the curve $\gamma$, and return back by Larmor arc. See Fig \ref{5}.
	\begin{figure}[h]
		\centering
		\includegraphics[width=0.5\linewidth]{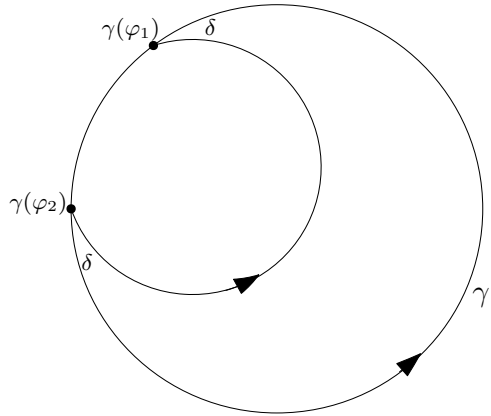}
		\caption{Larmor arc starting and ending with the same angle $\delta$.}
		\label{5}
	\end{figure}
	 We get:
	\begin{align*}
	&x(\varphi_2) - x(\varphi_1) = \int_{\bar{\varphi} - d(\bar{\varphi})}^{\bar{\varphi} + d(\bar{\varphi})}\rho(\xi)\cos(\xi)d\xi = -\frac{1}{\beta}\int^{\bar{\varphi}-d(\bar{\varphi})+2\pi-\delta}_{\bar{\varphi}+d(\bar{\varphi})+\delta}\cos(\xi)d\xi\\
	&y(\varphi_2) - y(\varphi_1) = \int_{\bar{\varphi} - d(\bar{\varphi})}^{\bar{\varphi} + d(\bar{\varphi})}\rho(\xi)\sin(\xi)d\xi = -\frac{1}{\beta}\int^{\bar{\varphi}-d(\bar{\varphi})+2\pi-\delta}_{\bar{\varphi}+d(\bar{\varphi})+\delta}\sin(\xi)d\xi
	\end{align*}
	Combining the above equations we get the desired equation (\ref{ro})
	\end{proof}
We wish to find a solution of this equation by a deformation of the solution corresponding to the unit circle. This means we assume:
$$
\rho(\varphi)=\rho_{\varepsilon}(\varphi); \  d(\varphi)=d_{\varepsilon}(\varphi),
$$ where for $\varepsilon=0$ we have
$$
\rho_0=1,\ d_0=const.
$$
Denote $$\rho_1(\varphi)= \frac { d }{d \varepsilon}  \Big |_{\varepsilon=0}\rho_{\varepsilon}(\varphi),\ 
d_1(\varphi)= \frac { d }{d \varepsilon}\Big  |_{\varepsilon=0}d_{\varepsilon}(\varphi)
$$

\begin{proposition}
	The following statements hold true: 
	
	(a)$\quad \delta, d_0,\beta$ are related by the equation:
	$$\beta\sin d_0=\sin(\delta+d_0).$$ 
(b)   	The Fourier coefficients  $\widehat{\rho_1}_n$ of $\rho_1$ and $\widehat{d_1}_n$ of $d_1$
	    all vanish except those $n$ which satisfy the Gutkin  relation: $$n \tan (d_0) = \tan(nd_0).$$
\end{proposition}

\begin{proof}	
Substituting in (\ref {ro}) $\varepsilon
=0$ we get (a).
Differentiating (\ref{ro}) with respect to $\varphi$ and extracting terms of order $\varepsilon$ we get
\begin{equation}\label{QUQU}
e^{id_0}\rho_1(\bar{\varphi}+d_0) -e^{-id_0}\rho_1(\bar{\varphi}-d_0) = 2\Big(\frac{1}{\beta}\cos(\delta+d_0)-\cos(d_0)\Big)\big(id_1(\bar{\varphi})+d_1'(\bar{\varphi}))\big)
\end{equation}
differentiate the imaginary part and subtract it from the real part we get 
\[
\cos(d_0)(\rho_1(\bar{\varphi}+d_0)-\rho_1(\bar{\varphi}-d_0)) = \sin(d_0)(\rho_1'(\bar{\varphi}+d_0)+\rho_1'(\bar{\varphi}-d_0))
\]
Equating Fourier coefficients in the last equation we get statement $(b)$.
\end{proof}

\begin{remark}
	In fact, it is an open question if the solution in formal power series in $\varepsilon$ of this equation can be constructed recursively and if the
	corresponding power series converge.
\end{remark}

\section{Proof of Proposition \ref{prop}  and Corollaries \ref{all1}, \ref{ellipse}}

\begin{proof}[Proof of Proposition \ref{prop}]
	Condition 2. of Definition \ref{def} reads in terms of $F$
	\begin{equation}\label{P}
	F(P_-(\epsilon))=F(P_+(\epsilon)).
	\end{equation}
	Differentiating this equality with respect to $\epsilon$ for
	$\epsilon=0$, and $\epsilon=\pi$ respectively and using the fact that
	$\frac{d}{d\epsilon}R_\epsilon =J\circ R_\epsilon$, we compute from (\ref{a})
	$$\frac{d}{d\epsilon}F(P_-(\epsilon))=dF(-rJR_{\pi/2-\epsilon}\dot\gamma(Q))=dF(-rR_{\pi-\epsilon}\dot\gamma(Q)),$$
	$$\frac{d}{d\epsilon}F(P_+(\epsilon))=dF(rJR_{\pi/2+\epsilon}\dot\gamma(Q))=dF(rR_{\pi+\epsilon}\dot\gamma(Q)).
	$$
	At these formulas for $\epsilon=0$ the differential  $dF$ is computed at the point $P'=P_-(0)=P_+(0)\in\gamma_{+ r}$ and is evaluated on the opposite vectors 	$\pm r\dot\gamma(Q)$.
	Similarly, for $\epsilon=\pi$ the differential  $dF$ is computed at the point $P''=P_-(\pi)=P_+(\pi)\in\gamma_{- r}$ and  is evaluated on the opposite vectors $\mp r\dot\gamma(Q)$ again.
	
	This proves that
	$$F|_{\gamma_{\pm r}}={\rm const}.$$
\end{proof}

\begin{proof}[Proof of Corollary \ref{all1}]
	Notice that $f_{\pm r}$ depends on $r$ as a polynomial function, so
	$f_{\pm r}$ is a polynomial in $x,y$, and $r$. If $r$ is small enough then all parallel curves $\gamma_{\pm r}$
	are smooth as real curves.
	However, there
	is an open interval of $r$, where, by using a differential geometry argument, one can claim that the fronts
	$\gamma_{+r}$ do have singularities. Hence, the system of
	equations
	$$
	\partial_x f_{+r}=\partial_y f_{+r}=f_{+r}=0
	$$
	defines an algebraic curve in ${\mathbb C}^3$ and its projection on
	the
	$r$-line is a Zariski open set.
	It then follows that singularities persist for all
	but finitely many $r.$
\end{proof}

It would be interesting to prove that every non-circular magnetic billiard does not admit a polynomial integral for all values of $\beta$.
This is out of reach at the present moment.

\begin{proof}[Proof of Corollary \ref{ellipse}]
	The equation of the parallel curves for the ellipse reads (see, e.g.,
	\cite{FN}):
	$$
	f_{+r}=f_{-r}=
	$$
	$$a^8 (b^4+(r^2-y^2)^2-2 b^2 (r^2+y^2))+b^4 (r^2-x^2)^2 (b^4-2 b^2 (r^2-x^2+y^2)+(x^2+y^2-r^2)^2)
	$$
	$$
	-2
	a^6 (b^6+(r^2-y^2)^2 (r^2+x^2-y^2)-b^4 (r^2-2 x^2+3 y^2)-b^2 (r^4+3 y^2 (x^2-y^2)+
	$$
	$$
	r^2 (3
	x^2+2 y^2)))+
	2 a^2 b^2 (-b^6 (r^2+x^2)-(-r^2+x^2+y^2)^2 (r^4-x^2 y^2-r^2 (x^2+y^2))+
	$$
	$$
	b^4 (r^4-3 x^4+3
	x^2 y^2+r^2 (2 x^2+3 y^2))+b^2 (r^6-2 x^6+x^4 y^2-3 x^2 y^4+r^4(-4 x^2+2 y^2)+
	$$
	$$
	r^2 (5 x^4-3 x^2 y^2-3 y^4)))+
	a^4 (b^8+2 b^6 (r^2+3 x^2-2 y^2)+(r^2-y^2)^2 (-r^2+x^2+y^2)^2-
	$$
	$$
	2 b^4 (3 r^4-3 x^4+5 x^2 y^2-3 y^4+4
	r^2 (x^2+y^2))+2 b^2 (r^6-3 x^4 y^2+x^2 y^4-2 y^6+
	$$
	$$
	2 r^4 (x^2-2 y^2)+r^2 (-3 x^4-3 x^2 y^2+5 y^4)))=0.
	$$
	It is known to be irreducible (see \cite{FN}). 
	Moreover, the
	parallel curves $\gamma_{\pm r}$ have singularities in the complex
	plane for every $r$ such that $\frac{1}{r}>2 k_{max}=\frac{2a}{b^2}$, which are
	$$
	(0,\pm\frac{\sqrt{b^2-a^2}\sqrt{a^2-r^2}}{a}),\quad (\pm\frac{\sqrt{a^2-b^2}\sqrt{b^2-r^2}}{b},0).
	$$
	Therefore the corollary follows from
	Theorem \ref{main1}.
\end{proof}

%%%%%%%%%%%%%%%%%%%%%%%%%%%%%%%%%%%%%%%%%%%%%%%%%%%%%%%%

\section {\bf Proof of main Theorem \ref{main1}}\label{mainproof}
We shall prove Theorem \ref{main1} for the curve $\gamma_{+ r}$ (proof for $\gamma_{- r}$ is completely analogous, replacing $P'$ and $\epsilon$ close to $0$ with $P''$ and $\epsilon$ close to $\pi$).

Our first step is to expand equation (\ref{P}) in powers of $\epsilon$ and to collect $\epsilon^3$ terms. 
For any function $F$ which is invariant under $\mathcal{M}$, we can
rewrite equation (\ref{P}) at any non-critical point
$P'\in\gamma_{+r}$ as follows:

Denote by $n=J\dot\gamma$ the unit normal vector to $\gamma$. Then $n$ is
also normal to $\gamma_{+r}$ and $\gamma_{-r}$ at the corresponding
points. It is useful to rewrite formulas (\ref{a})  in
a slightly more convenient form.  We have
\begin{equation}\label{a1}
P_\pm(\epsilon)=Q+rR_{\pi/2\pm\epsilon}\dot\gamma(Q)=Q+rR_{\pm\epsilon}n=
\end{equation}$$
=Q+rn-r(n-R_{\pm\epsilon}n) =P'-r(I-R_{\pm\epsilon})n
$$

Notice that  the unit normal $n$ of $\gamma$ is also a normal to the curve $\gamma_{+ r}$ at $P'$. One has $n=\pm\frac{\nabla F}{|\nabla F|}$.
We can rewrite equation (\ref{P}) with the help of
(\ref{a1}) :
\begin{align}\label{rem}
	& F\left(P'\pm r(I-R_{\epsilon})\left(\frac{\nabla F}{|\nabla F|}\right)(P')\right)- \\
	& F\left(P'\pm r(I-R_{-\epsilon})\left(\frac{\nabla F}{|\nabla F|}\right)(P')\right)=0 , \quad P'\in \gamma_{+r} \nonumber
\end{align}
This can be written for $P'=(x,y)\in\gamma_{+r}$ explicitly:

\begin{align}\label{rem1} 
	&F\left(x\pm r\frac{F_x(1-\cos\epsilon)+F_y\sin\epsilon}{|\nabla F|},y\pm r\frac{F_y(1-\cos\epsilon)-F_x\sin\epsilon}{|\nabla F|} \right)-\\
	&F\left(x\pm r\frac{F_x(1-\cos\epsilon)-F_y\sin\epsilon}{|\nabla F|},y\pm r\frac{F_y(1-\cos\epsilon)+F_x\sin\epsilon}{|\nabla F|} \right)=0\nonumber
\end{align}

The coefficient at $\epsilon^3$ reads
\begin{align}\label{rem2}
	&(F_{xxx}F_y^3-3F_{xxy}F_y^2F_x+3F_{xyy}F_yF_x^2-F_{yyy}F_x^3)\,\pm \\
	&3\beta(F_x^2+F_y^2)^{\frac{1}{2}}(F_{xx}F_xF_y+F_{xy}(F_y^2-F_x^2)-F_{yy}F_xF_y)=0,\quad (x,y)\in\gamma_{+ r} \nonumber
\end{align}

Remarkably, the left-hand side of (\ref{rem2}) is the complete
derivative along the tangent vector field $v$ to $\gamma_{+ r}$,
$v=(F_y,-F_x),$ of the following expression, which therefore must be
constant:
\begin{equation}\label{rem3}
H(F)\pm\beta|\nabla F|^3={\rm const}, \quad (x,y)\in\gamma_{+r};
\end{equation}
here we used the notation
$$
H(F):=F_{xx}F_y^2-2F_{xy}F_xF_y+F_{yy}F_x^2.
$$

Let us remark that (\ref{rem1}) and therefore also (\ref{rem3}) are
valid only for those points where $\nabla F$ does not vanish.
However,  this requirement is not satisfied if the
polynomial $F$ is not square-free. Therefore, we proceed as follows: Let
us denote by $f_{+r}$ irreducible defining polynomial of $\gamma_{+
	r}$. Then we have
$$ F=f_{+r}^k\cdot g,
$$ for some integer $k\geq 1$, where the polynomial $g$ does not
vanish identically on $\gamma_{+ r}$. Given an arc of $\gamma_{+ r}$
where $g$ does not vanish, we may assume that $g$ is positive on the
arc (otherwise we change the sign of $F$ and $f_{+r}$ if needed). Moreover, since $f_{+r}$
is an irreducible polynomial, we may assume that $\nabla f_{+r}$
does not vanish on the arc. Therefore equation (\ref{rem3}) can be
derived in the same manner for the function
$F^{\frac{1}{k}}=f_{+r}\cdot g^{\frac{1}{k}}$, which obviously is
invariant under the map $\mathcal{M}$ exactly as $F$ is.
Thus we
have
\begin{equation}\label{rem4}
H(f_{+r}\cdot g^{\frac{1}{k}})\pm \beta|\nabla (f_{+r}\cdot
g^{\frac{1}{k}})|^3={\rm const}, \quad (x,y)\in\{f_{+r}=0\}.
\end{equation}
Using the identities
$$
H(f_{+r}\cdot g^{\frac{1}{k}})=g^{\frac{3}{k}}H(f_{+r}),\quad \nabla
(f_{+r}\cdot g^{\frac{1}{k}})=g^{\frac{1}{k}}\nabla (f_{+r}),
$$ which are valid for all $(x,y)\in\{f_{+r}=0\}$.
We obtain from (\ref{rem4}) that
\begin{equation}\label{rem44}
g^{\frac{3}{k}}(H(f_{+r})\pm \beta|\nabla f_{+r}|^3)={\rm const}, \quad
(x,y)\in\gamma_{+r}.
\end{equation}
Raising back to the power $k$ we get
\begin{equation}\label{rem5}
g^3(H(f_{+r})\pm \beta|\nabla f_{+r}|^3)^k={\rm const}, \quad
(x,y)\in\gamma_{+r}.
\end{equation}
Next we establish the following:
\begin{proposition}\label{prop1}
	The constant in equation {\rm(\ref{rem5})}  cannot be $0$.
\end{proposition}
\begin{proof}
	Recall the formulas for the curvature $k$ of the curve defined
	implicitly by $\{f_{+r}=0\}$:
	\begin{equation}\label{k}
	\textrm{div} \left(\frac{\nabla f_{+r}}{|\nabla f_{+r}|}
	\right)=\frac{H(f_{+r})}{|\nabla f_{+r}|^3}=\pm k_{+r}.
	\end{equation}
	
	Now we take any point on $\gamma_{+r}$ and substitute into
	(\ref{rem5}). This gives that the constant must be non-zero. Indeed,
	if the constant is zero, then $$ \frac{H(f_{+r})}{|\nabla
		f_{+r}|^3}=\mp\beta.
	$$
	Then by formulas (\ref{k}) we have
	$$
	k_{+r}=\pm\beta.
	$$
	But this is not possible, because of the bounds on the curvature of
	the parallel curves (Lemma \ref{bounds}).
\end{proof}
%%%%%%%%%%%%%%%%%%%%%%%%%%%%%%%%%%%%%%%%%%

Now we are in position to complete the proof of Theorem \ref{main1}.
Consider the equation (\ref{rem5}) in $\mathbb{C}^2.$ It
follows from (\ref{rem5}) and Proposition \ref{prop1} that the curve
$\{f_{+r}=0\}$ has no singular points in $\mathbb{C}^2,$ since at
singular points both $H(f_{+r})$ and $\nabla (f_{+r})$ vanish.

This completes the proof of Theorem \ref{main1}.
\subsection{Infinite regular  points}
In fact we can say more on the points at infinity. 
\begin{theorem}
	Any non-singular
	point of intersection of the projective curve $\{\tilde{f}_{\pm
		r}=0\}$ in $\mathbb{C}P^2$ with the infinite line $\{z=0\}$ away
	from the isotropic points $(1:\pm i:0)$ must be a tangency point
	with the infinite line. Here $\tilde{f}_{\pm r}$ is a homogenization
	of $f_{\pm r}$.
\end{theorem}

\begin{proof} 
	Consider $\mathbb{C}P^2$ with homogeneous coordinates
$(x:y:z)$ and the projective curve $\{\tilde{f}_{+r}=0\}.$ We shall
denote the homogeneous polynomials corresponding to $f$ and $g$ by
$\tilde{f}$ and $\tilde{g}$, respectively. Then the homogeneous
version of (\ref{rem5}) for $(x:y:z)\in \{\tilde{f}_{+r}=0\}$ reads
\begin{equation}\label{hom}  \tilde{g}^3\left(z\cdot
H(\tilde{f}_{+r})+\beta((\tilde{f}_{+r})_x^2+
(\tilde{f}_{+r})_y^2)^\frac{3}{2}\right)^k={\rm const}\cdot z^p.
\end{equation}
Here the power $p=3\deg g+3k(\deg f_{+r}-1)$ must be positive unless
the degree of the polynomial $f_{+r}$ and that of $F$ are equal to
one. But this is impossible, due to our convexity assumptions. Let
$Z$ be any point of intersection of $ \{\tilde{f}_{+r}=0\}$ with the
infinite line $\{z=0\}.$ Then by (\ref{hom}) for such a point we
have the two relations
$$
(\tilde{f}_{+r})_x^2+ (\tilde{f}_{+r})_y^2=0,\quad
x(\tilde{f}_{+r})_x+y(\tilde{f}_{+r})_y+z(\tilde{f}_{+r})_z=x(\tilde{f}_{+r})_x
+y(\tilde{f}_{+r})_y=0.
$$
But these two relations are compatible only in two cases: either
$$x^2+y^2=z=0,$$
or $$(\tilde{f}_{+r})_x= (\tilde{f}_{+r})_y=0.$$ This completes the
proof.
\end{proof}
\section {Proofs of and Theorems \ref{smooth},\ref{harmonics}.}\label{proofs}

First we prove Theorem \ref{smooth}
\begin{proof}
	We need to prove smoothness at the interior points as well as at the boundary points. 
	%We use strongly that $F$ is correctly defined.
	
	1.  Let $y_0\in\Omega_r$ be an interior point. We want to show smoothness of $F$ in a neighborhood $U$ of $y_0$. Fix a unit vector $w_0$
	such that $$
	y_0=\mathcal L(x_0,w_0)=x_0+rJw_0
	,$$ for some $$x_0\in {\Omega_{+}}.$$
	For a fixed $w_0$ the formula
	$$
	y\leftrightarrow x, y=x+rJw_0 
	$$
	gives a diffeomorphism between $U$
	and a neighborhood of the point $x_0$.
	Since by definition 
	$$F(y)=F(x+rJw_0)=\Phi(x,w_0).
	$$
	The last function is smooth since the coefficients of $\Phi$ are assumed to be smooth.\\
	2. Assume now that $y_0$ is a boundary point. We distinguish between two cases:\\
	2A. Let $y_0\in\gamma_{+r}$. Then $y_0=\mathcal L(x_0,w_0)=x_0+rJw_0$, where $x_0=\gamma(s_0), w_0=\dot\gamma(s_0)$ (see Fig.\ref{7}).
		\begin{figure}[h]
			\centering
			\includegraphics[width=0.5\linewidth]{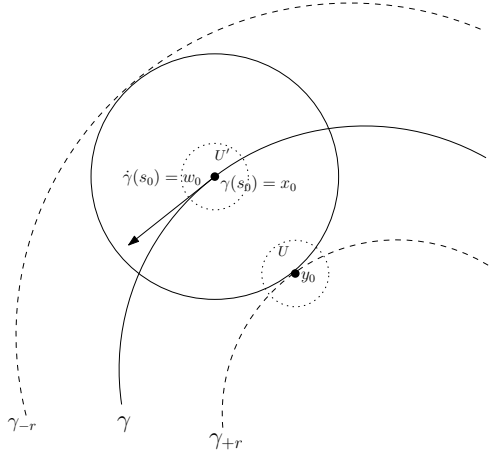}
			\caption{Smoothness of $F$ at the boundary point $y_0$.}
			\label{7}
		\end{figure}
		Let $U$ be a small ball around $y_0$. 
		Then the parallel transport 
		$$
		y\leftrightarrow x, y=x+rJw_0 
		$$ determines a diffeomorphism between the ball $U$ and a ball $U'$ around $x_0$.
		Exactly as we saw in previous item, we again get that
		$$F(y)=F(x+rJw_0)=\Phi(x,w_0).
		$$Thus $F$ extends smoothly to $U$ since $\Phi(x,w_0)$ is smooth on $U'$.\\
	2B. Proof in the case $y_0\in\gamma_{-r}$ is completely analogous. 
		This completes the proof of Theorem \ref{smooth}.
\end{proof}

In order to prove  Theorem \ref{harmonics} we need to localize the statement. 
For a small number $\epsilon$ we define the closed sets 
$$
A_{s_0,\epsilon}:=\bigcup_{s:|s-s_0|\leq\epsilon}C_s.
$$
Obviously the theorem follows from the following lemma:

\begin{lemma}\label{lemma}
	Let $F$ be a $C^{\infty}$ function, $F:A_{s_0,\epsilon}\rightarrow \mathbb{R}$. Suppose that the restriction of $F$
	 to any circle $C_s,\ |s-s_0|<\epsilon$ coincides with
	the restriction to $C_s$ of some polynomial $F_s$ of degree at most $N.$ It then follows that
	$F$ is a polynomial in $x_1,x_2$ of degree at most $2N.$
\end{lemma}

%%%%%%%%%%%%%%%%%%%%%%%%%%%%%%%%%%%%%%%%%%
%%%%%%%%%%%%%%%%%%%%%%%%%%%%%%%%%%%%%%%%%

\begin{proof} We shall say that $F$ has
	property $P_N$ if the restriction of $F$ to any circle $C_s,\ |s-s_0|<\epsilon$
	lying in $A_{s_0,\epsilon}$ coincides with
	the restriction to $C_s$ of some polynomial $F_s$ of degree at most $N$.
	The proof of the lemma goes by induction on the degree $N$.
	
	1) For $N=0$ the lemma holds since if $F$ has property $P_0$
	then $F$ is a constant on any circle $C_s$, and since any two of the circles $C_s$ intersect if $\epsilon<2r$, $F$ must be
		a constant on the whole $A_{s_0,\epsilon}$.
	
	2) Assume now that any function satisfying property $P_{N-1}$  is a
	polynomial of degree at most $2(N-1).$
	
	Let $F$ be any smooth function on $A_{s_0,\epsilon}$ with property $P_N$. Consider the circle
	$C_{s_0}$. We may assume that $\gamma({s_0})=0$ so the circle has the equation $$C_{s_0}=\{x_1^2+x_2^2=r^2\}.$$ Let $F_0$ be
	the polynomial of degree $N$ satisfying $F\big |_{C_{s_0}}=F_0\big |_{C_{s_0}}.$ 
	
	Next we apply 
	Hadamard's lemma to the function $F-F_0$. We consider the annulus 
	$$
	A=\{x: r-\delta\leq |x|\leq r+\delta\},
	$$
	where $\delta$ is small so that $F$ is smooth in $A$. Moreover, decreasing  $\epsilon$ if needed, we may assume that
	$$
	A_{s_0,\epsilon}\subset A.
	$$

	By Hadamard's lemma there exists a function $G\in C^{\infty}(A)$ so that
	\begin{equation}\label{F0}
	F(x_1,x_2)-F_0(x_1,x_2)=(x_1^2+x_2^2-r^2)G(x_1,x_2), \quad \forall (x_1,x_2)\in A.
	\end{equation}
	Let us show now that $G$ has property $P_{N-1}$ on $A_{s_0,\epsilon}$. Then by induction
	we will have that $G$ is a polynomial of degree $2(N-1)$ and thus
	$F$ is a polynomial of degree $2N$ at most. We need to show that the
	restriction $g:=G|_{C_s}$ to any circle $C_s$ coincides with the restriction of a polynomial of degree $(N-1)$ or less. With no loss of
	generality we may assume that the circle $C_s$ is centered on the
	$x-$axes (otherwise apply suitable rotation of the plane). Then
	$$C_s=\{(x_1,x_2)\in A_{s_0,\epsilon}: (x_1-a)^2+x_2^2=r^2\}, \quad |a|<\epsilon.$$
	Substituting $$x_1=a+r\cos t,\ \ x_2=r\sin t$$ into (\ref{F0}) we have
	$$
	(F-F_0)|_{C_s}=(a^2+2ar\cos t)\cdot g.
	$$
	Writing the left and the right hand side
	in Fourier series we get
	$$
	\sum_{-\infty}^{+\infty}f_ke^{ikt}=
	a(a+re^{it}+re^{-it})\sum_{-\infty}^{+\infty}g_ke^{ikt},
	$$ where the sum on the left hand side is finite, since $F$ has
	property $P_N.$ Thus we obtain for the coefficients:
	$$
	rg_{k+1}+ag_k+rg_{k-1}=\frac{1}{a}f_k, \quad\forall k\in \mathbb{Z}
	$$
	Since, we have: $$f_k=0,\quad \forall k, |k|>N,$$ then
	$$
	rg_{k+1}+ag_k+rg_{k-1}=0, \quad \forall k, |k|>N.
	$$
	The characteristic polynomial of this difference equation
	$$\lambda^2+\frac{a}{r}\lambda+1=0$$ has two complex conjugate roots since $|a|<\epsilon<2r$. Write the roots
	$$\lambda_{1,2}=e^{\pm i\alpha}, $$ therefore we get the formula:
	$$
	g_{N+l}=c_1e^{il\alpha}+c_2e^{-il\alpha}, \quad \forall l\geq2,
	$$ where
	$$c_1+c_2=g_N,\quad c_1e^{i\alpha}+c_2e^{-i\alpha}=g_{N+1}.
	$$
	It is obvious now that if at least one of the coefficients $g_N$ or
	$g_{N+1}$ does not vanish, then at least one of the constants $c_1,
	c_2$ does not vanish and therefore the sequence $\{g_{N+l}\}$ does
	not converge to $0$ when $l\rightarrow +\infty$. This contradicts
	the continuity of $g$. Therefore both $g_N,g_{N-1}$ must vanish and
	so $g$ is a trigonometric polynomial of degree at most $(N-1)$, proving that
	$G$ has property $P_{N-1}$. This completes the proof of Lemma.
\end{proof}

\end{document}